\documentclass[a4paper,12pt]{article}

\usepackage{amsmath}
\usepackage{amssymb}
\usepackage{amsthm}
\usepackage{graphicx}
\usepackage{enumerate}
\usepackage[colorlinks=true, linkcolor=blue,
            citecolor=green, urlcolor=blue]{hyperref}
            \usepackage{etoolbox}
\AtBeginDocument{\raggedbottom}

\newtheorem{theorem}{Theorem}[section]

\newtheorem{remark}[theorem]{Remark}

\theoremstyle{definition}
\newtheorem{definition}[theorem]{Definition}

\begin{document}
	
\title{On the Eigenvalues of the Biharmonic Steklov Problem on a Thin Set}
\author{
  \small Bauyrzhan Derbissaly \\
  \small Institute of Mathematics and Mathematical Modeling \\
  \small 125 Pushkin Street, Almaty, Kazakhstan \\
  \small \texttt{derbissaly@math.kz}
  \and
  \small Nurbek Kakharman \\
  \small SDU university\\
  \small    1/1 Abylaikhan street, Kaskelen, Kazakhstan\\\small Institute of Mathematics and Mathematical Modeling \\
  \small 125 Pushkin Street, Almaty, Kazakhstan \\
  \small \texttt{n.kakharman@math.kz;~~nurbek.kakharman@sdu.edu.kz}
}
\date{}
\maketitle


\begin{abstract}
\noindent
This paper investigates the asymptotic behavior of the eigenvalues of the biharmonic operator on a thin set with Steklov boundary condition. The thin set is taken to be a tubular neighborhood of a planar smooth domain. We show that, as the thickness of this neighborhood tends to zero, all eigenvalues of the biharmonic operator with Steklov boundary condition converge to zero.
\end{abstract}


\begin{flushleft} \small
\end{flushleft}


\section{\large Introduction and statement of the main result}\label{sect. 1}
\noindent
We are interested in analyzing the asymptotic behavior of the eigenvalues of the biharmonic Steklov problem 
on a thin set. For $\varepsilon>0$, we define the domain $\omega_\varepsilon$ as
\[
\omega_\varepsilon := \{x \in \Omega : \text{dist}(x,\partial \Omega) < \varepsilon\},
\]
where $\Omega$ is a bounded domain in $\mathbb{R}^2$ with smooth boundary $\partial\Omega$.
In $\omega_\varepsilon$ we consider the following Steklov problem for the biharmonic operator
\begin{equation}\label{2}
\begin{cases}
\Delta^2 u_\varepsilon = 0, 
&\text{in }\omega_\varepsilon,\\
\partial_{\nu \nu}^2u_\varepsilon=-\mu\partial_{\nu} u_\varepsilon, 
&\text{on }\partial\omega_\varepsilon,\\
-\text{div}_{\partial\omega_\varepsilon}\left(D^2u_\varepsilon\cdot\nu\right)_{\partial\omega_\varepsilon}-\partial_{\nu}\Delta u_\varepsilon=\lambda_\varepsilon u_\varepsilon,
&\text{on }\partial\omega_\varepsilon,
\end{cases}
\end{equation}
in the unknowns $u_\varepsilon$ (the Steklov eigenfunction) and $\lambda_\varepsilon$ (the Steklov eigenvalue).  
Here, $\mu>0$ is fixed constant, and $\nu$ denotes the outer unit normal to $\omega_\varepsilon$, $D^2 u_\varepsilon$ denotes the Hessian of $u_\varepsilon$, $\text{div}_{\partial \omega_\varepsilon}$ denotes the tangential divergence on $\partial \Omega$, and $(D^2 u_\varepsilon \cdot \nu)_{\partial \omega_\varepsilon}$ denotes the projection of $D^2 u_\varepsilon \cdot \nu$ on the tangent space $T \partial \omega_\varepsilon$.

A weak formulation of problem \eqref{2} has the form
\[
\int_{\omega_\varepsilon}D^2u_\varepsilon:D^2\varphi dx+\mu\int_{\partial\omega_\varepsilon}\partial_\nu u_\varepsilon\partial_\nu\varphi d\sigma =
\lambda_\varepsilon\int_{\partial\omega_\varepsilon}u_\varepsilon\varphi d\sigma,    
\]
for all $\varphi\in H^2(\omega_\varepsilon)$. By
$
D^2u : D^2v:=
\sum_{i,j=1}^2
\partial_{x_ix_j}^2 u\partial_{x_ix_j}^2 v
$
we denote the standard product of Hessians.

Since \(\partial\Omega\) is smooth, there exists \(\varepsilon_0>0\) such that for all \(\varepsilon\in(0,\varepsilon_0)\), the set \(\omega_\varepsilon\) is also a smooth domain. In that range, problem \eqref{2} is well‐posed and admits a discrete sequence of nonnegative eigenvalues diverging to \(+\infty\):
\[
0=\lambda_{\varepsilon,0}<\lambda_{\varepsilon,1}\leq\ \cdot\cdot\cdot\leq \lambda_{\varepsilon,n}\leq...
\] 
with corresponding eigenfunctions \(u_{\varepsilon,k}\) normalized in \(L^2(\partial\omega_\varepsilon)\).

The goal of this paper is to describe the asymptotic behavior of the solutions \((u_{\varepsilon,k},\,\lambda_{\varepsilon,k})\) of \eqref{2} as \(\varepsilon\to0\). When \(\varepsilon\) is small, we refer to \(\omega_\varepsilon\) as a “thin domain”, which collapses to the curve \(\partial\Omega\) as \(\varepsilon\to0\).

In the limiting case \(\mu=+\infty\), problem \eqref{2} reduces to the (NBS)‐Neumann Biharmonic Steklov problem:
\[
\begin{cases}
\Delta^2 u_\varepsilon = 0, 
&\text{in }\omega_\varepsilon,\\
\partial_{\nu} u_\varepsilon=0, 
&\text{on }\partial\omega_\varepsilon,\\
-\text{div}_{\partial\omega_\varepsilon}\left(D^2u_\varepsilon\cdot\nu\right)_{\partial\omega_\varepsilon}-\partial_{\nu}\Delta u_\varepsilon=\lambda_\varepsilon u_\varepsilon,
&\text{on }\partial\omega_\varepsilon.
\end{cases}
\]
Recall that the case \(\mu=0\) in \eqref{2} was introduced in \cite{buopro} as a natural fourth‐order extension of the classical Steklov problem for the Laplacian.

Our study is motivated by \cite{ferpro}, where the authors examined the asymptotic behavior of eigenvalues for the biharmonic operator with Neumann boundary conditions on the same thin set \(\omega_\varepsilon\). They proved that, as \(\varepsilon\to0\), those eigenvalues converge to the eigenvalues of a limiting system of differential equations. In contrast, under the Steklov boundary conditions in \eqref{2}, the limiting problem reduces to a single differential equation, and all eigenvalues tend to zero. To facilitate comparison with \cite{ferpro}, we follow a similar strategy in our analysis.

In classical plate theory a Steklov boundary condition appears when the inertia of the body is concentrated along the edge: for transverse vibrations of a free plate with line mass density on $\partial\Omega$, separation in time shows that the boundary traction must be proportional to the trace of the displacement, and the proportionality constant becomes the spectral parameter on $\partial\Omega$; this is exactly the biharmonic Steklov model (see, \cite{buopro}).

Beyond classical elasticity, an analogous idea applies in micropolar (Cosserat) elasticity in the sense of Eringen \cite{Eringen1968}, \cite{Eringen1999}. There the primary fields are the displacement $\mathbf u$ and the microrotation $\mathbf \vartheta$, the stresses are accompanied by couple-stresses, and the natural boundary data are the traction and the couple-traction. The theory also includes a (bulk) microinertia accounting for rotational kinetic energy of the microstructure. If, in the same spirit as above, one models a thin structure whose edge carries (part of) the translational line mass and rotational microinertia, the time-harmonic balance on $\partial\Omega$ leads to Steklov-type laws on the boundary in which the spectral parameter multiplies the traces of $(\mathbf u,\mathbf \vartheta)$. Physically, the Steklov condition then encodes that a non-negligible fraction of kinetic energy-translational and rotational-resides on the boundary, and eigenmodes are selected by the competition between bulk bending energy and this edge inertia. We record this micropolar interpretation only as motivation; our analysis below concerns the biharmonic case with $\sigma=0$.

The asymptotic analysis of eigenvalue problems on singularly perturbed domains, particularly thin domains, has been a subject of intense study in spectral theory and partial differential equations, see, for example, \cite{JM}, \cite{VP}, \cite{FP},  \cite{CD}, \cite{GA}, \cite{JC}, \cite{SA}, \cite{MC}, \cite{Miura}, \cite{Li}, \cite{Nogueira}, \cite{Antonio}, \cite{Nakasato} and the references therein. A classical body of work concerns the Laplace operator with various boundary conditions, where the transition from a thin domain to its limiting lower-dimensional object is often achieved via a delicate analysis of the associated energy forms, see e.g., \cite{Henry,G} for Neumann conditions, and \cite{Grieser} for Dirichlet and Robin conditions. The common theme is the identification of a limiting operator on the collapsed structure whose spectrum captures the asymptotic behavior of the original eigenvalues.

Throughout this paper, we set the Poisson ratio \(\sigma=0\), as in \cite{ferpro}. In two dimensions, \(\sigma\in(-1,1)\) in general, and choosing \(\sigma\ne0\) would alter the quadratic form appearing in the weak formulation \eqref{2}. Although one could treat the more general case \(\sigma\in(-1,1)\), passing to the limit as \(\varepsilon\to0\) becomes substantially more involved (see \cite{PDL}). Hence, we restrict our attention to the emblematic case \(\sigma=0\).

The main result of this paper is the following theorem:

\begin{theorem}\label{thm:1.1}
Let \(\{\lambda_{\varepsilon,k}\}_{k\in\mathbb{N}^+}\) be the sequence of eigenvalues of problem \eqref{2}. Then, as \(\varepsilon\to0\),
\[
\lambda_{\varepsilon,k}
\;\sim\;\lambda_k\,\varepsilon,
\]
where \(\lambda_k\) is the \(k\)‐th eigenvalue of the following one‐dimensional problem
\begin{equation}\label{3}
\begin{cases}
u^{(4)}-2\bigl(\kappa^2u'\bigr)'=2\lambda u,
&\text{in }(0,|\partial\Omega|),\\
u^{(k)}(0)=u^{(k)}(|\partial\Omega|),
&k=0,1,2,3,
\end{cases}
\end{equation}
with unknown \(u=u(s)\) and eigenvalue \(\lambda\). Here \(s\) denotes the arclength parameter on \(\partial\Omega\), and \(\kappa(s)\) is the curvature at the point \(s\in(0,|\partial\Omega|)\). Moreover, there exists an orthonormal basis \(\{u_k\}_{k\in\mathbb{N}^+}\) of eigenfunctions of \eqref{3} in \(L^2\bigl((0,|\partial\Omega|),\,2\,ds\bigr)\) such that, possibly after passing to a subsequence,
\[
u_{\varepsilon,k} \circ \Phi_{\varepsilon}\rightharpoonup u_k 
\ \text{in }H^2(\Sigma)
\ \text{as } \varepsilon\to0,
\]
where \(u_k\) depends only on the variable \(s\).
\end{theorem}
The mappings \(\Phi_\varepsilon\) and the reference set \(\Sigma\) are defined in Section~\ref{sect. 2}.


\section{\large Preliminaries and notation}\label{sect. 2}
\subsection{\large The functional-geometric setting}
\noindent
Let $\Omega$ be a bounded domain in $\mathbb{R}^2$ of class $C^{0,1}$ (i.e. Lipschitz). We denote by $H^k\left(\Omega\right)$ the standard Sobolev space of all functions 
$u\in L^2\left(\Omega\right)$ for which every weak derivative 
$D^\alpha u$ of order $|\alpha|\leq k$
exists and also lies in 
$L^2\left(\Omega\right)$.

The Sobolev space $H^2\left(\Omega\right)$ is naturally endowed with the norm 
\[
\|u\|_{H^2\left(\Omega\right)}=\left(\|D^2u\|^2_{L^2\left(\Omega\right)}+\|u\|^2_{L^2\left(\Omega\right)}\right)^{\frac{1}{2}}.
\]
In \cite{PDLLP} it is pointed out that 
\[
\|u\|_{H^2\left(\Omega\right)}=\left(\|D^2u\|^2_{L^2\left(\Omega\right)}+\|u\|^2_{L^2\left(\partial\Omega\right)}\right)^{\frac{1}{2}}.
\] 
is equivalent to the standard norm of $H^2\left(\Omega\right)$.

By 
$
H^k_p\bigl((0,|\partial\Omega|)\bigr)
$
we denote the closure in 
$
H^k\bigl((0,|\partial\Omega|)\bigr)
$
of the space 
$
C^\infty_p\bigl((0,|\partial\Omega|)\bigr),
$
which consists of those functions in \(C^\infty\bigl((0,|\partial\Omega|)\bigr)\) satisfying
$
u^{(k)}(0)=u^{(k)}(|\partial\Omega|)
\ \text{for all }k\in\mathbb{N}.
$

Now we establish the geometry of a tubular neighborhood  of smooth boundaries and an adapted local coordinate system. Recall the following result from \cite{fed}.
\begin{theorem}\cite{fed}\label{thm:2.1}
Let \(k\ge2\) and let \(\Omega\subset\mathbb{R}^2\) be a bounded domain of class \(C^k\). Then there exists \(\varepsilon>0\) such that every point in \(\omega_\varepsilon\) has a unique nearest point on \(\partial\Omega\). Moreover, the function \(\text{dist}(\cdot,\partial\Omega)\) is of class \(C^k\) in \(\omega_\varepsilon\).
\end{theorem}

Throughout the rest of the paper, denote by \(\varepsilon_0\) the maximal tubular radius of \(\Omega\), namely
\[
\varepsilon_0 := \sup \bigl\{\varepsilon>0 : \text{every point in }\omega_\varepsilon \text{ has a unique nearest point on }\partial\Omega \bigr\}.
\]
By Theorem~\ref{thm:2.1}, if \(\partial\Omega\) is \(C^k\), then \(\varepsilon_0>0\). For any \(x\in\omega_\varepsilon\) with nearest boundary point \(y\in\partial\Omega\), one has
\[
1 -\mathrm{dist}(x,\partial\Omega)\kappa(y) > 0,
\]
where \(\kappa(y)\) is the curvature of \(\partial\Omega\) at \(y\) with respect to the outward normal (see, e.g., \cite[Lemma~2.2]{lll}).

Fix a base point \(x_0 \in \partial\Omega\) and identify \(\partial\Omega\) with the interval \([0,|\partial\Omega|]\) via the arclength parameter \(s\), with \(s=0\) and \(s=|\partial\Omega|\) corresponding to the same point \(x_0\). Write \(\nu(s)\) for the outward unit normal at the boundary point parameterized by \(s\).

Define the coordinate map
\[
\Phi_\varepsilon : \Sigma := \partial\Omega \times (0,1)
\rightarrow \omega_\varepsilon,
\
\Phi_\varepsilon(s,t):= s- \varepsilon t \nu(s).
\]
Since \(\partial\Omega\) is \(C^k\) and \(\varepsilon<\varepsilon_0\), \(\Phi_\varepsilon\) is a diffeomorphism onto \(\omega_\varepsilon\) (see, e.g., \cite[Section~2.4]{bel}). Identifying \(\partial\Omega\) with \((0,|\partial\Omega|)\), we view \(\Phi_\varepsilon\) as a diffeomorphism from \((0,|\partial\Omega|)\times(0,1)\) onto \(\omega_\varepsilon\). In these coordinates, 
\[
t = \varepsilon^{-1}\mathrm{dist}\bigl(s - \varepsilon t \nu(s),\partial\Omega\bigr).
\]
The pair \((s,t)\) is often called \emph{curvilinear} or \emph{Fermi coordinates}. In \(\omega_\varepsilon\), they form a global coordinate system. In particular, for any integrable function \(f\) on \(\omega_\varepsilon\),
\begin{equation}\label{6}
\int_{\omega_\varepsilon} f(x)\,dx 
=
\int_{\Sigma} \bigl(f \circ \Phi_\varepsilon\bigr)(s,t)\varepsilon\bigl(1 - \varepsilon t \kappa(s)\bigr) dt ds.
\end{equation}

Moreover, if \(f,g\in H^2(\omega_\varepsilon)\), then \(f\circ\Phi_\varepsilon\) and \(g\circ\Phi_\varepsilon\) lie in \(H^2(\Sigma)\). A lengthy but standard calculation (see \cite[Section~2.2, formula~(2.7)]{ferpro}) shows that
\begin{equation}\label{7}
\begin{split}
(D^2 f :& D^2 g)\circ \Phi_\varepsilon
\\=& \frac{\partial^2_{ss}(f\circ\Phi_\varepsilon)\partial^2_{ss}(g\circ\Phi_\varepsilon)}{(1 - \varepsilon t \kappa(s))^4}
+ \frac{2\partial^2_{st}(f\circ\Phi_\varepsilon)\partial^2_{st}(g\circ\Phi_\varepsilon)}{\varepsilon^2(1 - \varepsilon t \kappa(s))^2}
+ \frac{\partial^2_{tt}(f\circ\Phi_\varepsilon)\partial^2_{tt}(g\circ\Phi_\varepsilon)}{\varepsilon^4}\\
&
+ \frac{\varepsilon t\kappa'(s)}{(1 - \varepsilon t \kappa(s))^5}
\bigl(\partial_s(f\circ\Phi_\varepsilon)\,\partial^2_{ss}(g\circ\Phi_\varepsilon)
+ \partial^2_{ss}(f\circ\Phi_\varepsilon)\,\partial_s(g\circ\Phi_\varepsilon)\bigr)\\
&
- \frac{\kappa(s)}{\varepsilon(1 - \varepsilon t \kappa(s))^3}
\bigl(\partial^2_{ss}(f\circ\Phi_\varepsilon)\partial_t(g\circ\Phi_\varepsilon)
+ \partial_t(f\circ\Phi_\varepsilon)\,\partial^2_{ss}(g\circ\Phi_\varepsilon)\bigr)\\
&
+ \frac{2\kappa(s)}{\varepsilon(1 - \varepsilon t \kappa(s))^3}
\bigl(\partial_s(f\circ\Phi_\varepsilon)\partial^2_{st}(g\circ\Phi_\varepsilon)
+ \partial^2_{st}(f\circ\Phi_\varepsilon)\partial_s(g\circ\Phi_\varepsilon)\bigr)\\
&
- \frac{t\kappa(s)\kappa'(s)}{(1 - \varepsilon t \kappa(s))^4}
\bigl(\partial_s(f\circ\Phi_\varepsilon)\partial_t(g\circ\Phi_\varepsilon)
+ \partial_t(f\circ\Phi_\varepsilon)\partial_s(g\circ\Phi_\varepsilon)\bigr)\\
&
+ \frac{2\kappa(s)^2(1 - \varepsilon t \kappa(s))^2 + \varepsilon^2 t^2(\kappa'(s))^2}
     {(1 - \varepsilon t \kappa(s))^6}
  \partial_s(f\circ\Phi_\varepsilon)\partial_s(g\circ\Phi_\varepsilon)\\
&
+ \frac{\kappa(s)^2}{\varepsilon^2(1 - \varepsilon t \kappa(s))^2}
  \partial_t(f\circ\Phi_\varepsilon)\partial_t(g\circ\Phi_\varepsilon).
\end{split}
\end{equation}

Similarly, for the normal derivative we have
\begin{equation}\label{8}
(\nabla f\cdot\nu)\circ\Phi_\varepsilon(s,0)
=
-\frac{1}{\varepsilon}\partial_t(f\circ\Phi_\varepsilon)(s,0),
\quad (\nabla f\cdot\nu)\circ\Phi_\varepsilon(s,1)
=\frac{1}{\varepsilon}\partial_t(f\circ\Phi_\varepsilon)(s,1).
\end{equation}

Recall that if \(f,g\in H^2(\omega_\varepsilon)\), then 
$
f\circ\Phi_\varepsilon, g\circ\Phi_\varepsilon \in H^2(\Sigma),
$
since \(\Phi_\varepsilon\) is smooth. By a standard density argument, identity \eqref{7} holds for all \(f,g\in H^2(\omega_\varepsilon)\).

Throughout the sequel we write 
$
\tilde f := f\circ\Phi_\varepsilon \in H^2(\Sigma)
$ for the pull-back of any function 
\(f\in H^2(\omega_\varepsilon)\).


\subsection{\large Modified problems}
\noindent
Let $\varepsilon\in(0,\varepsilon_0)$. Referring to the left‐hand side of the weak formulation of problem \eqref{2}, we introduce a shifted eigenvalue  $\overline{\lambda}_\varepsilon:=\lambda_\varepsilon+\varepsilon b$, where \(b>0\) is a constant to be chosen later.  We then consider the modified Steklov problem
\begin{equation}\label{9}
\begin{cases}
\Delta^2 u_\varepsilon = 0, 
&\text{in }\omega_\varepsilon,\\
\partial_{\nu \nu}^2u_\varepsilon=-\mu\partial_{\nu} u_\varepsilon, 
&\text{on }\partial\omega_\varepsilon,\\
-\text{div}_{\partial\omega_\varepsilon}\left(D^2u_\varepsilon\cdot\nu\right)_{\partial\omega_\varepsilon}-\partial_{\nu}\Delta u_\varepsilon+\varepsilon b u_\varepsilon =\overline{\lambda}_\varepsilon u_\varepsilon,
&\text{on }\partial\omega_\varepsilon.
\end{cases}
\end{equation}
Its variational formulation reads:
\begin{equation}\label{10}
\int_{\omega_\varepsilon}D^2u_\varepsilon:D^2\varphi dx+\mu\int_{\partial\omega_\varepsilon}\partial_\nu u_\varepsilon\partial_\nu\varphi d\sigma+\varepsilon b\int_{\partial\omega_\varepsilon}u_\varepsilon\varphi d\sigma=
\overline{\lambda}_\varepsilon\int_{\partial\omega_\varepsilon}u_\varepsilon\varphi d\sigma,  
\end{equation}
for all $\varphi\in H^2(\omega_\varepsilon)$.
Clearly, analyzing the asymptotic behavior of \(\lambda_\varepsilon\) as \(\varepsilon\to0\) is equivalent to studying \(\overline{\lambda}_\varepsilon\).

Next, fix a boundary datum \(f_\varepsilon\in L^2(\partial\omega_\varepsilon)\) and consider the corresponding inhomogeneous problem:
\begin{equation}\label{2.4}
\begin{cases}
\Delta^2 u_\varepsilon = 0, 
&\text{in }\omega_\varepsilon,\\
\partial_{\nu \nu}^2u_\varepsilon=-\mu\partial_{\nu} u_\varepsilon, 
&\text{on }\partial\omega_\varepsilon,\\
-\text{div}_{\partial\omega_\varepsilon}\left(D^2u_\varepsilon\cdot\nu\right)_{\partial\omega_\varepsilon}-\partial_{\nu}\Delta u_\varepsilon+\varepsilon b u_\varepsilon =f_\varepsilon,
&\text{on }\partial\omega_\varepsilon.
\end{cases}
\end{equation}
Its weak formulation is
\begin{equation}\label{12}
\int_{\omega_\varepsilon}D^2u_\varepsilon:D^2\varphi dx+\mu\int_{\partial\omega_\varepsilon}\partial_\nu u_\varepsilon\partial_\nu\varphi d\sigma+\varepsilon b\int_{\partial\omega_\varepsilon}u_\varepsilon\varphi d\sigma=
\int_{\partial\omega_\varepsilon}f_\varepsilon\varphi d\sigma,      
\end{equation}
for all $\varphi\in H^2(\omega_\varepsilon)$.

Similarly, we introduce a shifted version of the one‐dimensional limiting problem \eqref{3}.  Namely, we define
\begin{equation}\label{13}
\begin{cases}
u^{(4)}-2\bigl(\kappa^2u'\bigr)'+2b u=2\lambda u,
&\text{in }(0,|\partial\Omega|),\\
u^{(k)}(0)=u^{(k)}(|\partial\Omega|),
&k=0,1,2,3.
\end{cases}
\end{equation}
Given a datum \(f\in L^2\bigl(0,|\partial\Omega|\bigr)\), the corresponding inhomogeneous limit problem is
\begin{equation}\label{14}
\begin{cases}
u^{(4)}-2\bigl(\kappa^2u'\bigr)'+2b u=2f,
&\text{in }(0,|\partial\Omega|),\\
u^{(k)}(0)=u^{(k)}(|\partial\Omega|),
&k=0,1,2,3,
\end{cases}
\end{equation}
whose weak formulation is
\[
\int^{|\partial\Omega|}_{0}u''\phi''+2\kappa^2u'\phi'+2bu\phi ds=   2\int^{|\partial\Omega|}_{0}f\phi ds, 
\]
for all $\phi\in H_p^2((0,|\partial\Omega|))$.


\subsection{\large Convergence of operators and their spectra}\label{ss:2.4}
\noindent
In this subsection, we recall a few definitions of convergence of operators as well as related results on spectral convergence.

We note that the domains $\omega_\varepsilon$ vary with $\varepsilon$, and hence when considering the biharmonic Steklov problem in $\omega_\varepsilon$, the corresponding Hilbert spaces $\mathcal{H}_\varepsilon$ depend on $\varepsilon$, complicating the direct application of the standard notion of compact convergence. So, we shall employ suitable connecting systems that facilitate the transition from the variable Hilbert spaces to the fixed limiting Hilbert space. 

This methodology integrates various concepts and results from the works of Stummel \cite{FS} and Vainikko \cite{GMV}, which have been further elaborated in \cite{JMA, AC}. Notably, we utilize the concept of $E$-compact convergence.

In the spirit of \cite{JMA}, we denote by $\mathcal{H}_\varepsilon$ a family of Hilbert spaces for $\varepsilon\in [0, \varepsilon_0]$ and assume the existence of a family of linear operators $E_\varepsilon: \  \mathcal{H}_0 \rightarrow \mathcal{H}_\varepsilon$ such that
\begin{equation}\label{16}
\|E_\varepsilon f\|_{\mathcal{H}_\varepsilon}\xrightarrow{\varepsilon\rightarrow 0}\|f\|_{\mathcal{H}_0}, \ \text{for all} \ f\in\mathcal{H}_0.
\end{equation}

\begin{definition}\label{def 2.1}
We say that a family $\{f_\varepsilon\}_{0<\varepsilon\leq\varepsilon_0}$, with $f_\varepsilon\in \mathcal{H}_\varepsilon, E$-converges to $f\in \mathcal{H}_0$ if $\|f_\varepsilon-E_\varepsilon f\|_{\mathcal{H}_\varepsilon}\rightarrow0$
as $\varepsilon\rightarrow0$. We write this as $f_\varepsilon\xrightarrow{E}f$.
\end{definition}

\begin{definition}\label{def 2.2}
Let $\{B_\varepsilon\in\mathcal{L}\left(\mathcal{H}_\varepsilon\right): \ \varepsilon\in(0,\varepsilon_0]\}$ be a family of linear and continuous operators. We say
that $\{B_\varepsilon\}_{0<\varepsilon\leq\varepsilon_0}$ $E$-converges to $B_0\in\mathcal{L}\left(\mathcal{H}_0\right)$ as $\varepsilon\rightarrow0$ if $B_\varepsilon f_\varepsilon\xrightarrow{E}B_0f$ whenever $f_\varepsilon\xrightarrow{E}f$. We write this as
$B_\varepsilon\xrightarrow{EE}B_0$.
\end{definition}

\begin{definition}\label{def 2.3}
Let $\{f_\varepsilon\}_{0<\varepsilon\leq\varepsilon_0}$ be a family such that $f_\varepsilon\in \mathcal{H}_\varepsilon$. We say that $\{f_\varepsilon\}_{0<\varepsilon\leq\varepsilon_0}$ is precompact if
for any sequence $\varepsilon_n\rightarrow0$ there exist a subsequence $\{\varepsilon_{n_k}\}_{k\in\mathbb{N}}$ and $f\in \mathcal{H}_0$ such that $f_{\varepsilon_{n_k}}\xrightarrow{E}f$ as $k\to\infty$.
\end{definition}
\begin{definition}\label{def 2.4}
We say that $\{B_\varepsilon\}_{0<\varepsilon\leq\varepsilon_0}$ with $B_\varepsilon\in\mathcal{L}\left(\mathcal{H}_\varepsilon\right)$ and $B_\varepsilon$ compact, converges compactly to a
compact operator $B_0\in\mathcal{L}\left(\mathcal{H}_0\right)$ if $B_\varepsilon\xrightarrow{EE}B_0$ and for any family $\{f_\varepsilon\}_{0<\varepsilon\leq\varepsilon_0}$ such that $f_\varepsilon\in \mathcal{H}_\varepsilon$, $\|f_\varepsilon\|_{\mathcal{H}_\varepsilon}=1$,
we have that $\{B_\varepsilon f_\varepsilon\}_{0<\varepsilon\leq\varepsilon_0}$ is precompact in the sense of Definition \ref{def 2.3}. We write this as $B_\varepsilon\xrightarrow{C}B_0$.
\end{definition}

The $E$-compact convergence implies spectral stability. Namely, we have the following result where by `generalized eigenfunction' associated to $m$ eigenvalues  we mean a linear combination of $m$ eigenfunctions associated to those egenvalues.

\begin{theorem}\label{th 2.5}
Let $\{B_\varepsilon\}_{0\le \varepsilon\leq\varepsilon_0}$ be a family of non-negative, compact self-adjoint operators in the
Hilbert spaces $\mathcal{H}_\varepsilon$. Assume that their eigenvalues are given by $\{\lambda_k(\varepsilon)\}^{\infty}_{k=1}$. If $B_\varepsilon\xrightarrow{C}B_0$, then there is spectral convergence of $B_\varepsilon$ to $B_0$ as $\varepsilon\rightarrow0$.
In particular, the following statements hold: 
\begin{itemize}
 \item[(i)] For every $k\in  \mathbb{N}$ we have $\lambda_k(\varepsilon )\to \lambda_k(0)$ as $\varepsilon \to 0$.
 \item[(ii)] If $u_k(\varepsilon)$, $k\in \mathbb{N}$, is an orthonormal sequence of eigenfunctions associated with the eigenvalues  $\lambda_k(\varepsilon )$ then
 there exists an orthonormal sequence of eigenfunctions  $u_k(0)$, $k\in  \mathbb{N}$ associated with  $\lambda_k(0 )$, $k\in  \mathbb{N}$ such that, possibly passing to a subsequence, $u_n(\varepsilon )\xrightarrow{E} u_k(0)$.
 \item[(iii)]  Given  $m$ eigenvalues $\lambda_k(0),  \dots , \lambda_{k+m-1}(0)$ with
$\lambda_k(0)\ne \lambda_{k-1}(0)$ and $\lambda_{k+m-1}(0)$ $\ne \lambda_{k+m}(0)$
 and corresponding orthonormal eigenfunctions $u_k(0),\dots,u_{k+m-1}(0)$
 there exist $m$ orthonormal   generalized eigenfunctions   $v_k(\varepsilon ),  \dots , v_{k+m-1}(\varepsilon )$  associated with  $\lambda_k(\varepsilon)$,  $\dots ,  
  \lambda_{k+m-1}(\varepsilon )$   such that $v_{k+j}(\varepsilon )\xrightarrow{E} u_{k+j}(0)$  for all $j=0, 1,\dots , m-1$.
 \end{itemize}    
\end{theorem}
We refer to \cite[Theorem~2.5]{AFPDL}, \cite[Theorem~6.3]{GMV}, \cite[Theorem~1]{SS}, see also \cite[Theorem~4.10]{JMA}, \cite[Theorem~5.1]{EA} and
\cite[Theorem~3.3]{AC} 
for more details on spectral convergence.


\section{\large Proof of Theorem \ref{thm:1.1}}\label{sect. 3}
\subsection{\large A key coercivity estimate}
\noindent
First, observe that inequality \eqref{12} can be reformulated in the $(s,t)$-coordinates on $(0,|\partial\Omega|)\times(0,1)$ (see \eqref{6}, \eqref{8}) as
\begin{equation}\label{17}
\begin{aligned}
\int_{\Sigma}(D^2u_\varepsilon&:D^2\varphi) \circ\Phi_\varepsilon(s,t)\,
(1 - \varepsilon t\kappa(s))dtds
\\+&\frac{\mu}{\varepsilon^3}\int^{|\partial\Omega|}_{0}\partial_t(u_\varepsilon\circ\Phi_\varepsilon)(s,0)\partial_t(\varphi\circ\Phi_\varepsilon)(s,0)ds
\\+&\frac{\mu}{\varepsilon^3}\int^{|\partial\Omega|}_{0}\partial_t(u_\varepsilon\circ\Phi_\varepsilon)(s,1)\partial_t(\varphi\circ\Phi_\varepsilon)(s,1)(1 - \varepsilon\kappa(s))ds
\\+& b\int^{|\partial\Omega|}_{0}(u_\varepsilon\varphi)\circ\Phi_\varepsilon(s,0)ds
+b\int^{|\partial\Omega|}_{0}(u_\varepsilon\varphi)\circ\Phi_\varepsilon(s,1)(1 - \varepsilon\kappa(s))ds
\\&=\frac{1}{\varepsilon}
\int^{|\partial\Omega|}_{0}(f_\varepsilon\varphi)\circ\Phi_\varepsilon(s,0)ds
+\frac{1}{\varepsilon}\int^{|\partial\Omega|}_{0}(f_\varepsilon\varphi)\circ\Phi_\varepsilon(s,1)(1-\varepsilon\kappa(s))ds,
\end{aligned}
\end{equation}  
for all $\varphi\in H^2(\Sigma)$. We assume that 
$\sup_{\varepsilon\in(0,\varepsilon_0)}
\|\varepsilon^{-1}\tilde f_\varepsilon(\cdot,i)\|_{{L^2((0,|\partial\Omega|)}}\ne \infty,  (i=0,1),
$
and
$
\varepsilon^{-1}\tilde f_\varepsilon(s,0) \rightharpoonup f_1(s) \ \text{in} \ L^2\left((0,|\partial\Omega|)\right),   
$ and $
\varepsilon^{-1}\tilde f_\varepsilon(s,1) \rightharpoonup f_2(s) \ \text{in} \ L^2\left((0,|\partial\Omega|)\right)  
$ as $\varepsilon\to 0$. Let us denote the left-hand side of \eqref{17} by $a(\tilde u_\varepsilon, \varphi)$.

From \eqref{7} we obtain the following expansion for the integrand $D^2u_\varepsilon : D^2\varphi$ in the new coordinates:
\begin{align}
(D^2u_\varepsilon &: D^2\varphi)\circ\Phi_\varepsilon
= \label{18}\\&
\frac{\partial^2_{ss}(u_\varepsilon\circ\Phi_\varepsilon)\partial^2_{ss}(\varphi\circ\Phi_\varepsilon)}
        {(1 - \varepsilon t\kappa(s))^4}
+ \frac{2\partial^2_{st}(u_\varepsilon\circ\Phi_\varepsilon)\partial^2_{st}(\varphi\circ\Phi_\varepsilon)}
         {\varepsilon^2(1 -  \varepsilon t\kappa(s))^2}
+ \frac{\partial^2_{tt}(u_\varepsilon\circ\Phi_\varepsilon)\partial^2_{tt}(\varphi\circ\Phi_\varepsilon)}
         {\varepsilon^4}
   \label{19}\\
&
+ \frac{\varepsilon t \kappa'(s)}{(1 - \varepsilon t\kappa(s))^5}
  (\partial_s(u_\varepsilon\circ\Phi_\varepsilon) \partial^2_{ss}(\varphi\circ\Phi_\varepsilon)
        + \partial^2_{ss}(u_\varepsilon\circ\Phi_\varepsilon) \partial_s(\varphi\circ\Phi_\varepsilon))
   \label{20}\\
&
- \frac{\kappa(s)}{\varepsilon(1 - \varepsilon t\kappa(s))^3}
  (\partial^2_{ss}(u_\varepsilon\circ\Phi_\varepsilon)\partial_t(\varphi\circ\Phi_\varepsilon)
        + \partial_t(u_\varepsilon\circ\Phi_\varepsilon)\partial^2_{ss}(\varphi\circ\Phi_\varepsilon))
   \label{21}\\
&
+ \frac{2\kappa(s)}{\varepsilon (1 - \varepsilon t\kappa(s))^3}
  (\partial_s(u_\varepsilon\circ\Phi_\varepsilon)\partial^2_{st}(\varphi\circ\Phi_\varepsilon)
        + \partial^2_{st}(u_\varepsilon\circ\Phi_\varepsilon)\partial_s(\varphi\circ\Phi_\varepsilon))
   \label{22}\\
&
- \frac{t \kappa(s) \kappa'(s)}{(1 - \varepsilon t \kappa(s))^4}
(\partial_s(u_\varepsilon\circ\Phi_\varepsilon) \partial_t(\varphi\circ\Phi_\varepsilon)
        + \partial_t(u_\varepsilon\circ\Phi_\varepsilon) \partial_s(\varphi\circ\Phi_\varepsilon))
   \label{23}\\
&
+ \frac{2\kappa(s)^2(1 - \varepsilon t \kappa(s))^2 + \varepsilon^2 t^2(\kappa'(s))^2}{(1 - \varepsilon t \kappa(s))^6}
\partial_s(u_\varepsilon\circ\Phi_\varepsilon)\partial_s(\varphi\circ\Phi_\varepsilon)\label{24}\\
&
+ \frac{\kappa(s)^2}{\varepsilon^2 (1 - \varepsilon t\kappa(s))^2}
\partial_t(u_\varepsilon\circ\Phi_\varepsilon)\partial_t(\varphi\circ\Phi_\varepsilon).
   \label{25}
\end{align}

Let $\tilde u_\varepsilon = u_\varepsilon\circ\Phi_\varepsilon$. Now, we prove that there exists a constant $c>0$ such that, for all $\varepsilon\in(0,\varepsilon_0/2)$,
\begin{equation}\label{26}
\begin{aligned}
\frac{\|\partial^2_{tt}\tilde u_\varepsilon\|^2_{L^2(\Sigma)}}{\varepsilon^4}
+&\frac{\|\partial^2_{st}\tilde u_\varepsilon\|^2_{L^2(\Sigma)}}{\varepsilon^2}
+\frac{\|\kappa\partial_t\tilde u_\varepsilon\|^2_{L^2(\Sigma)}}{\varepsilon^2}\\
+&\|\partial^2_{ss}\tilde u_\varepsilon\|^2_{L^2(\Sigma)}
+\|\kappa\partial_s\tilde u_\varepsilon\|^2_{L^2(\Sigma)}
+\|\tilde u_\varepsilon(\cdot,0)\|^2_{{L^2((0,|\partial\Omega|))}}+\|\tilde u_\varepsilon(\cdot,1)\|^2_{{L^2((0,|\partial\Omega|))}}
\\ &\le  c\|\varepsilon^{-1}\tilde f_\varepsilon(\cdot,0)\|^2_{L^2((0,|\partial\Omega|))}
+c\|\varepsilon^{-1}\tilde f_\varepsilon(\cdot,1)\|^2_{L^2((0,|\partial\Omega|))}.
\end{aligned}
\end{equation}

Next, we set
\[
\varrho(s,t):=1 - \varepsilon t \kappa(s).
\]
Since $\varepsilon\le\varepsilon_0/2$, there exist constants $0<c_1\le c_2$ (independent of $\varepsilon$) such that
\[
c_1\le\varrho(s,t)\le c_2,\ \forall \ (s,t)\in\Sigma.
\]
We now take the test function $\varphi = u_\varepsilon$ in \eqref{17}. Inspecting the terms in \eqref{19}, one finds that
\[
\bigl\|\varrho^{-3/2}\,\partial^2_{ss}\tilde u_\varepsilon\bigr\|^2_{L^2(\Sigma)},\
2\bigl\|\varepsilon^{-1}\varrho^{-1/2}\partial^2_{st}\tilde u_\varepsilon\bigr\|^2_{L^2(\Sigma)},\
\bigl\|\varrho^{1/2}\varepsilon^{-2}\partial^2_{tt}\tilde u_\varepsilon\bigr\|^2_{L^2(\Sigma)},
\]
while from \eqref{24} and \eqref{25} one gets
\[
2\bigl\|\kappa \varrho^{-3/2}\partial_s\tilde u_\varepsilon\bigr\|^2_{L^2(\Sigma)}+\bigl\|\varepsilon t \kappa'\varrho^{-5/2}\partial_s\tilde u_\varepsilon\bigr\|^2_{L^2(\Sigma)},
\
\bigl\|\varepsilon^{-1}\varrho^{-1/2}\kappa \partial_t\tilde u_\varepsilon\bigr\|^2_{L^2(\Sigma)},
\]
respectively. The remaining mixed derivatives in \eqref{20} must be estimated using Cauchy–Schwarz and standard interpolation inequalities. For example, for any $\delta>0$ and sufficiently small $\theta>0$, one has
\begin{equation}\label{27}
\begin{aligned}
\int_\Sigma &\frac{2\varepsilon t \kappa'(s)}{\varrho^5} 
  \partial_s\tilde u_\varepsilon \partial^2_{ss}\tilde u_\varepsilon \varrho dt ds
\ge -c\int_\Sigma\bigl|\partial_s\tilde u_\varepsilon \partial^2_{ss}\tilde u_\varepsilon\bigr| dt\,ds\\
&\ge
-\frac{c\delta}{2} \|\partial^2_{ss}\tilde u_\varepsilon\|^2_{L^2(\Sigma)}
  -\frac{c}{2\delta}\|\partial_s\tilde u_\varepsilon\|^2_{L^2(\Sigma)} \\
&\ge 
-\frac{c\delta}{2} \|\partial^2_{ss}\tilde u_\varepsilon\|^2_{L^2(\Sigma)}
  -\frac{c\theta}{2\delta}\|\partial^2_{ss}\tilde u_\varepsilon\|^2_{L^2(\Sigma)}
 -\frac{c c_0}{2\delta \theta}\|\tilde u_\varepsilon\|^2_{L^2(\Sigma)},
\end{aligned}    
\end{equation}
where
\[
c=\max_{\substack{(s,t)\in\Sigma,\\0\le \varepsilon\le \varepsilon_0/2}}
\frac{\varepsilon t\lvert\kappa'(s)\rvert}{(1 - \varepsilon t\kappa(s))^4},
\]
and in passing to the last line we applied the one‐dimensional interpolation
\begin{equation}\label{28}
\|v'\|^2_{L^2((a,b))}
\le
\theta\|v''\|^2_{L^2((a,b))}
+
\frac{c_0}{\theta} \|v\|^2_{L^2((a,b))},    
\end{equation}
with some constant $c_0>0$. Moreover, we combine
\begin{equation}\label{29}
\|v'\|^2_{L^2((a,b))}
\le
\overline{\theta}_0\|v''\|^2_{L^2((a,b))}
+
\frac{\overline{c}_1}{\overline{\theta}_0}\left(|v(a)|^2+ |v(b)|^2\right), 
\end{equation}
and
\begin{equation}\label{30}
\|v\|^2_{L^2((a,b))}
\le
\tilde{\theta}_0\|v'\|^2_{L^2((a,b))}
+
\frac{\tilde{c}_1}{\tilde{\theta}_0}\left(|v(a)|^2+|v(b)|^2\right),   
\end{equation}
to deduce
\begin{equation}\label{31}
\begin{aligned}
\|\tilde u_\varepsilon\|^2_{L^2(\Sigma)}
\le&
\tilde{\theta}_0\|\partial_t\tilde u_\varepsilon\|^2_{L^2(\Sigma)}
+
\frac{\tilde{c}_1}{\tilde{\theta}_0}\int_0^{|\partial \Omega|}\left(|\tilde u_\varepsilon(s,0)|^2+|\tilde u_\varepsilon(s,1)|^2\right)ds
\\ \le&
\overline{\theta}_0\tilde{\theta}_0\|\partial^2_{tt}\tilde u_\varepsilon\|^2_{L^2(\Sigma)}
+
\frac{\tilde{c}_1}{\tilde{\theta}_0}\int_0^{|\partial \Omega|}\left(|\tilde u_\varepsilon(s,0)|^2+|\tilde u_\varepsilon(s,1)|^2\right)ds
\\&+
\frac{\overline{c}_1}{\overline{\theta}_0}\int_0^{|\partial \Omega|}\left(|\tilde u_\varepsilon(s,0)|^2+|\tilde u_\varepsilon(s,1)|^2\right)ds
\\&=
\theta^2_0\|\partial^2_{tt}\tilde u_\varepsilon\|^2_{L^2(\Sigma)}
+
\frac{c_1}{\theta_0}\int_0^{|\partial \Omega|}\left(|\tilde u_\varepsilon(s,0)|^2+|\tilde u_\varepsilon(s,1)|^2\right)ds,
\end{aligned}
\end{equation}
where we have set $c_1:=\tilde{c}_1+\overline{c}_1$ and $\theta_0=\tilde{\theta}_0=\overline{\theta}_0$. Combining \eqref{27} and \eqref{31} yields
\begin{equation}\label{32}
\begin{aligned}
\int_\Sigma &\frac{2\varepsilon t \kappa'(s)}{\varrho^5} 
  \partial_s\tilde u_\varepsilon \partial^2_{ss}\tilde u_\varepsilon \varrho dt ds
 \\
&\ge 
-\frac{c(\delta^2+\theta)}{2\delta} \|\partial^2_{ss}\tilde u_\varepsilon\|^2_{L^2(\Sigma)}
 -\frac{c c_0\theta^2_0}{2\delta \theta}\frac{\|\partial^2_{tt}\tilde u_\varepsilon\|^2_{L^2(\Sigma)}}{\varepsilon^4} 
 \\&
- \frac{c c_0c_1}{\theta_0\theta}\int_0^{|\partial \Omega|}|\tilde u_\varepsilon(s,0)|^2ds
-\frac{c c_0c_1c'}{\theta_0\theta}\int_0^{|\partial \Omega|}|\tilde u_\varepsilon(s,1)|^2(1 - \varepsilon\kappa(s))ds,
\end{aligned}
\end{equation}
where
\[
c' =\frac{1}{\displaystyle\min_{\substack{s\in[0,|\partial \Omega|],\\0\le \varepsilon\le\varepsilon_0/2}}
(1 - \varepsilon\kappa(s))}.
\]

By similar arguments applied to \eqref{22} and \eqref{23}, one also obtains (redefining constants if necessary)
\begin{equation}\label{33}
 \begin{aligned}
\int_{\Sigma}
\frac{4\kappa(s)}{\varepsilon\varrho^3}&
\partial_{s}\tilde u_{\varepsilon}\partial^2_{st}\tilde u_{\varepsilon} \varrho dtds
 \\
\ge &
-\frac{c\delta}{2} \frac{\|\partial^2_{st}\tilde u_\varepsilon\|^2_{L^2(\Sigma)}}{\varepsilon^2}-\frac{c\theta}{2\delta} \|\partial^2_{ss}\tilde u_\varepsilon\|^2_{L^2(\Sigma)}
 -\frac{c c_0\theta^2_0}{2\delta \theta}\frac{\|\partial^2_{tt}\tilde u_\varepsilon\|^2_{L^2(\Sigma)}}{\varepsilon^4} 
 \\&
- \frac{c c_0c_1}{\theta_0\theta}\int_0^{|\partial \Omega|}|\tilde u_\varepsilon(s,0)|^2ds
-\frac{c c_0c_1c'}{\theta_0\theta}\int_0^{|\partial \Omega|}|\tilde u_\varepsilon(s,1)|^2(1 - \varepsilon\kappa(s))ds,
 \end{aligned}   
\end{equation}
\begin{equation}\label{34}
 \begin{aligned}
\int_{\Sigma}
\frac{2t\kappa(s)\kappa'(s)}{\varrho^4}&
\partial_{s}\tilde u_{\varepsilon}\partial_{t}\tilde u_{\varepsilon}\varrho dtds
 \\
\ge &
-\frac{c\delta}{2} \frac{\|\kappa\partial_{t}\tilde u_\varepsilon\|^2_{L^2(\Sigma)}}{\varepsilon^2}-\frac{c\theta}{2\delta} \|\partial^2_{ss}\tilde u_\varepsilon\|^2_{L^2(\Sigma)}
 -\frac{c c_0\theta^2_0}{2\delta \theta}\frac{\|\partial^2_{tt}\tilde u_\varepsilon\|^2_{L^2(\Sigma)}}{\varepsilon^4} 
 \\&
- \frac{c c_0c_1}{\theta_0\theta}\int_0^{|\partial \Omega|}|\tilde u_\varepsilon(s,0)|^2ds
-\frac{c c_0c_1c'}{\theta_0\theta}\int_0^{|\partial \Omega|}|\tilde u_\varepsilon(s,1)|^2(1-\varepsilon\kappa(s))ds.
 \end{aligned}   
\end{equation}
Here \(\delta,\theta_0, \theta>0\) can be chosen arbitrarily small and independent of \(\varepsilon\), and \(c, c_0, c_1, c'\) are positive constants not depending on \(\varepsilon\).

The most delicate contribution arises from \eqref{21}. Integrating by parts in $s$ and again using \eqref{28}-\eqref{30} leads to
\begin{equation}\label{35}
 \begin{aligned}
\int_{\Sigma}
\frac{2\kappa(s)}{\varepsilon\varrho^3}
&\partial^2_{ss}\tilde u_{\varepsilon}\partial_{t}\tilde u_{\varepsilon}
\varrho dtds
\\ \ge&
-2\Biggl\lvert
\int_{\Sigma}
\frac{\kappa(s)}{\varrho^2}
\partial_{s}\tilde u_{\varepsilon}
\frac{\partial^2_{st}\tilde u_{\varepsilon}}{\varepsilon}
+
\partial_{s}\left(\frac{\kappa(s)}{\varrho^2}\right)
\partial_{s}\tilde u_{\varepsilon}
\frac{\partial_{t}\tilde u_{\varepsilon}}{\varepsilon}
dsdt
\Biggr\rvert
\\
\ge&
-\frac{c\delta}{2} \frac{\|\partial^2_{st}\tilde u_\varepsilon\|^2_{L^2(\Sigma)}}{\varepsilon^2}-\frac{c\theta}{2\delta} \|\partial^2_{ss}\tilde u_\varepsilon\|^2_{L^2(\Sigma)}
 -\frac{c c_0\theta^2_0}{2\delta \theta}\frac{\|\partial^2_{tt}\tilde u_\varepsilon\|^2_{L^2(\Sigma)}}{\varepsilon^4} 
 \\&
- \frac{c c_0c_1}{\theta_0\theta}\int_0^{|\partial \Omega|}|\tilde u_\varepsilon(s,0)|^2ds
-\frac{c c_0c_1c'}{\theta_0\theta}\int_0^{|\partial \Omega|}|\tilde u_\varepsilon(s,1)|^2(1 - \varepsilon\kappa(s))ds
\\&
-\frac{c\delta'}{2} \frac{\|\partial_{t}\tilde u_\varepsilon\|^2_{L^2(\Sigma)}}{\varepsilon^2}-\frac{c\theta'}{2\delta'} \|\partial^2_{ss}\tilde u_\varepsilon\|^2_{L^2(\Sigma)}
 -\frac{c c_0\theta^2_0}{2\delta' \theta'}\frac{\|\partial^2_{tt}\tilde u_\varepsilon\|^2_{L^2(\Sigma)}}{\varepsilon^4} 
 \\&
- \frac{c c_0c_1}{\theta_0\theta'}\int_0^{|\partial \Omega|}|\tilde u_\varepsilon(s,0)|^2ds
-\frac{c c_0c_1c'}{\theta_0\theta'}\int_0^{|\partial \Omega|}|\tilde u_\varepsilon(s,1)|^2(1-\varepsilon\kappa(s))ds,
\end{aligned}   
\end{equation}
where $c,c_0,c_1,c'>0$ are independent of $\varepsilon$, and $\delta,\delta',\theta,\theta',\theta_0>0$ remain at our disposal. The goal is to make the coefficients in front of all derivative‐terms small, at the expense of possibly enlarging the coefficient multiplying the boundary terms. Ultimately, one recognizes that the first‐order $t$–derivative term in \eqref{35} must be bounded from below by
\[
\frac{\|\kappa\partial_{t}\tilde u_{\varepsilon}\|_{L^{2}(\Sigma)}^{2}}{\varepsilon^{2}}.
\]
Since $\kappa$ may vanish on a subset of $\partial\Omega$ of positive measure, we cannot directly equate these norms. However, by the Gauss–Bonnet Theorem,
\[
\int_{\partial\Omega} \kappa(s)ds =2\pi;
\]
so there exists an open subset $J_a \subset \partial\Omega$ and a constant $a>0$ such that $\lvert\kappa(s)\rvert\ge a$ for all $s\in J_a$. In particular, one shows (cf.\ \cite{ferpro}) that there is a constant $c_a>0$ satisfying
\begin{equation}\label{36}
\frac{\|\partial_{t}\tilde u_{\varepsilon}\|_{L^{2}(\Sigma)}^{2}}{\varepsilon^{2}}
\le
4c_{a}^{2}
\frac{\|\partial^{2}_{tt}\tilde u_{\varepsilon}\|_{L^{2}(\Sigma)}^{2}}{\varepsilon^{2}}
+
4c_{a}^{2}
\frac{\|\partial^{2}_{st}\tilde u_{\varepsilon}\|_{L^{2}(\Sigma)}^{2}}{\varepsilon^{2}}
+
\frac{2}{a^{2}}
\frac{\|\kappa\partial_{t}\tilde u_{\varepsilon}\|_{L^{2}(\Sigma)}^{2}}{\varepsilon^{2}}    
\end{equation}
where the proof can be found in \cite{ferpro}. Then, we inserting \eqref{36} into \eqref{35}.

Lastly, one estimates the right‐hand side of \eqref{17} with $\varphi = u_\varepsilon$. By the Young inequality,
\begin{equation}\label{37}
\begin{aligned}
\frac{1}{\varepsilon}
\int^{|\partial\Omega|}_{0}&\tilde f_\varepsilon(s,0) \tilde u_\varepsilon (s,0)ds
+\frac{1}{\varepsilon}\int^{|\partial\Omega|}_{0}\tilde f_\varepsilon(s,1) \tilde u_\varepsilon(s,1)(1-\varepsilon\kappa(s))ds
\\ \le&
\frac{1}{\varepsilon^2}
\int^{|\partial\Omega|}_{0}\tilde f^2_\varepsilon(s,0)ds
+\frac{1}{2}
\int^{|\partial\Omega|}_{0}\tilde u^2_\varepsilon(s,0)ds
\\&
+\frac{c'}{\varepsilon^2}
\int^{|\partial\Omega|}_{0}\tilde f^2_\varepsilon(s,1)ds
+\frac{1}{2}
\int^{|\partial\Omega|}_{0}\tilde u^2_\varepsilon(s,1)(1-\varepsilon\kappa(s))ds.
\end{aligned}    
\end{equation}

Collecting all estimates \eqref{32}–\eqref{35} and choosing $\delta,\delta',\theta,\theta',\theta_0>0$ sufficiently small (and possibly increasing $b$ if needed), one finds constants $c_0>0$ and $c_1>1$, independent of $\varepsilon$, such that
\begin{equation}\label{38}
\begin{aligned}
a(\tilde u_{\varepsilon}, &\tilde u_{\varepsilon})\ge
c_{0}\Biggl(
\frac{\bigl\|\partial_{tt}^{2}\tilde u_{\varepsilon}\bigr\|^{2}_{L^{2}(\Sigma)}}{\varepsilon^{4}}
+
\frac{\bigl\|\partial_{st}^{2}\tilde u_{\varepsilon}\bigr\|^{2}_{L^{2}(\Sigma)}}{\varepsilon^{2}}
+
\frac{\bigl\|\kappa \partial_{t}\tilde u_{\varepsilon}\bigr\|^{2}_{L^{2}(\Sigma)}}{\varepsilon^{2}}
\Biggr)
\\&+
c_{0}\Bigl(
\bigl\|\partial_{ss}^{2}\tilde u_{\varepsilon}\bigr\|^{2}_{L^{2}(\Sigma)}
 + 
\bigl\|\kappa\partial_{s}\tilde u_{\varepsilon}\bigr\|^{2}_{L^{2}(\Sigma)}
\Bigr)
 \\&+ 
c_{1} \int^{|\partial\Omega|}_{0}\tilde u^2_\varepsilon(s,0)ds+c_1\int^{|\partial\Omega|}_{0}\tilde u^2_\varepsilon(s,1)(1-\varepsilon\kappa(s))ds,
\end{aligned}
\end{equation}
for all \(\varepsilon\in[0,\varepsilon_0/2]\). Then by \eqref{37}, we finally deduce that \eqref{26} holds. This completes the proof of the coercivity estimate.


\subsection{\large Derivation of the limiting problem}\label{ss:3.2}
\noindent
The coercivity estimate \eqref{26} shows that the families
\[
\Bigl\{\frac{\partial_{tt}^{2}\tilde u_{\varepsilon}}{\varepsilon^{2}}\Bigr\}_{\varepsilon\in(0,\varepsilon_0)},\
\Bigl\{\frac{\partial_{st}^{2}\tilde u_{\varepsilon}}{\varepsilon}\Bigr\}_{\varepsilon\in(0,\varepsilon_0)},\
\{\partial_{ss}^{2}\tilde u_{\varepsilon}\}_{\varepsilon\in(0,\varepsilon_0)}  
\]
are uniformly bounded in \(L^{2}(\Sigma)\) for all \(\varepsilon\in(0,\varepsilon_0)\).  In particular, \(\{\tilde u_{\varepsilon}\}_{\varepsilon\in(0,\varepsilon_0)}\) is a bounded sequence in \(H^{2}(\Sigma)\).  Hence, there exists \(u\in H^{2}(\Sigma)\) and a subsequence (still indexed by \(\varepsilon\)) such that
\[
\tilde u_{\varepsilon}\rightharpoonup u
\ \text{in }H^{2}(\Sigma),
\ \text{and hence strongly in }H^{1}(\Sigma).
\]
Moreover, there exists a function \(v\in L^{2}(\Sigma)\) such that, up to a subsequence,
\begin{equation}\label{40}
\frac{\partial_{tt}^{2}\tilde u_{\varepsilon}}{\varepsilon^{2}}
\ \rightharpoonup  v
\ \text{in }L^{2}(\Sigma)
\ \text{as }\varepsilon\to 0.
\end{equation}

Next, from \eqref{26} (or more precisely from \eqref{36} and \eqref{38}) it follows that the sequence
\[
\Bigl\{\frac{\partial_{t}\tilde u_{\varepsilon}}{\varepsilon}\Bigr\}_{\varepsilon\in(0,\varepsilon_0)}
\]
is uniformly bounded in \(L^{2}(\Sigma)\).  Therefore, up to a subsequence, there exists a function \(w\in H^{1}(\Sigma)\) such that
\[
\frac{\partial_{t}\tilde u_{\varepsilon}}{\varepsilon}
\rightharpoonup w
\quad\text{in }H^{1}(\Sigma)
\quad\text{as }\varepsilon\to 0.
\]
By the compact embedding \(H^{1}(\Sigma)\hookrightarrow L^{2}(\Sigma)\), one also has
\[
\frac{\partial_{t}\tilde u_{\varepsilon}}{\varepsilon}
\rightarrow w
\ \text{in }L^{2}(\Sigma)
\ \text{as }\varepsilon\to 0.
\]
In particular, \(\partial_{t}\tilde u_{\varepsilon}\to 0\) in \(L^{2}(\Sigma)\).  Hence \(\partial_{t}u = 0\) a.e. in \(\Sigma\), so \(u\) is constant in the \(t\)-variable.

We further deduce that the limiting function \(w\) is constant in \(t\), due to the fact that
\[
\frac{\partial^2_{tt}\tilde u_{\varepsilon}}{\varepsilon} \rightarrow 0 
\ \text{in }L^2(\Sigma).
\]

On the other hand, from \eqref{17} we have
\[
\frac{1}{\varepsilon^3}\int^{|\partial\Omega|}_{0}(\partial_t\tilde u_\varepsilon(s,0))^2ds\leq c.
\]
Then the following inequality
\[
\int^{|\partial\Omega|}_{0}\left(\frac{\partial_t\tilde u_\varepsilon(s,0)}{\varepsilon}\right)^2ds\leq \varepsilon,
\]
yields that 
\[
\biggl\|\frac{\partial_{t}\tilde u_{\varepsilon}(s,0)}{\varepsilon}\biggr\|_{L^2((0,|\partial\Omega|))} \rightarrow 0 \ \text{as }\varepsilon\to 0.
\]
So we have $w(s,0)=0$ a.e. on $(0,|\partial\Omega|)$. Being $t$ independent of $w$ forces $w\equiv0$ on $\Sigma$.

We now pass to the limit in the weak formulation \eqref{17}.  First, take 
\[
\varphi = \varepsilon^{2}\,\zeta
\quad\text{for some }\zeta\in H^{2}(\Sigma).
\]
All terms in \eqref{17} vanish as \(\varepsilon\to 0\), except possibly
\[
\frac{1}{\varepsilon^{4}}
\partial_{tt}^{2}\bigl(u_{\varepsilon}\circ \Phi_{\varepsilon}\bigr)
\partial_{tt}^{2}\bigl(\varepsilon^{2}\zeta\bigr).
\]
From \eqref{40} and \eqref{17}, we then deduce that
\[
\int_{\Sigma}
\bigl(D^{2}u_{\varepsilon} : D^{2}(\varepsilon^{2}\zeta)\bigr)\circ \Phi_{\varepsilon}(s,t)
(1-\varepsilon t\kappa(s))dtds
\rightarrow
\int_{\Sigma} v\partial_{tt}^{2}\zeta ds dt
=0
\ \text{as }\varepsilon\to 0.
\]
Since \(\zeta\) is an arbitrary function in \(H^{2}(\Sigma)\), we conclude that \(v=0\).

Next, let
\[
\varphi(s,t)=\psi(s),
\ s\in\bigl(0,|\partial\Omega|\bigr),\
\psi\in H^{2}_{p}\bigl(0,|\partial\Omega|\bigr).
\]
Using \(\varphi\) as a test function in \eqref{17}, we deduce that
\[
\begin{split}
\int_{\Sigma}\Biggl(&
\frac{\partial^2_{ss}\tilde u_{\varepsilon}\partial^2_{ss}\varphi}{(1-\varepsilon t\kappa(s))^4}
+
\frac{\varepsilon t\kappa'(s)}{(1 - \varepsilon t\kappa(s))^5}\bigl(\partial_{s}\tilde u_{\varepsilon}\partial^2_{ss}\varphi
+\partial^2_{ss}\tilde u_{\varepsilon}\partial_{s}\varphi\bigr)\\
&
-\frac{\kappa(s)}{\varepsilon(1-\varepsilon t\kappa(s))^3}\partial_{t}\tilde u_{\varepsilon}\partial^2_{ss}\varphi
+
\frac{2\kappa(s)}{\varepsilon(1-\varepsilon t\kappa(s))^3}\partial^2_{st}\tilde u_{\varepsilon}\partial_{s}\varphi
\\&
-\frac{t\kappa(s)\kappa'(s)}{(1-\varepsilon t\kappa(s))^4}\partial_{t}\tilde u_{\varepsilon}\partial_{s}\varphi
+
\frac{2\kappa(s)^2(1-\varepsilon t\kappa(s))^2 + \varepsilon^2t^2\kappa'(s)^2}{(1-\varepsilon t\kappa(s))^6}
\partial_{s}\tilde u_{\varepsilon}\partial_{s}\varphi
\Biggr)
(1 - \varepsilon t\kappa(s))dsdt
\\&+
b\int^{|\partial\Omega|}_{0}\tilde u_\varepsilon(s,0)\varphi(s,0)ds
+b\int^{|\partial\Omega|}_{0}\tilde u_\varepsilon(s,1)\varphi(s,1)(1 - \varepsilon\kappa(s))ds
\\&=\frac{1}{\varepsilon}
\int^{|\partial\Omega|}_{0}\tilde f_\varepsilon(s,0)\varphi(s,0)ds
+\frac{1}{\varepsilon}\int^{|\partial\Omega|}_{0}\tilde f_\varepsilon(s,1)\varphi(s,1)(1-\varepsilon\kappa(s))ds.
\end{split}
\]
Passing to the limit \(\varepsilon\to 0\), using the weak convergences
\begin{equation}\label{42}
\int_{\Sigma}\left(
\partial^2_{ss}u\partial^2_{ss}\psi
+2\kappa(s)^{2}\partial_{s}u\partial_{s}\psi
\right)dsdt+2b\int^{|\partial\Omega|}_{0}u\psi ds
=
\int^{|\partial\Omega|}_{0}(f_1+f_2)\psi ds,
\end{equation}
and noting that \(u,\psi\) do not depend on \(t\), we obtain
\begin{equation}\label{43}
\begin{cases}
u^{(4)}-2\bigl(\kappa^2u'\bigr)'+2bu=f_1+f_2,
&\text{in }(0,|\partial\Omega|),\\
u^{(k)}(0)=u^{(k)}(|\partial\Omega|),
&k=0,1,2,3.
\end{cases}
\end{equation}
In particular, if \(f_{1}=f_{2}=f\), then \eqref{43} reduces exactly to problem \eqref{14} stated earlier.


\subsection{\large Proof of the compact convergence}
\noindent
First, we introduce the Hilbert spaces required in Section \ref{ss:2.4}. For every \(\varepsilon\in(0,\varepsilon_0)\), we define
\[
\mathcal{H}_\varepsilon=L^{2}\left(\partial\omega_{\varepsilon};\varepsilon^{-2}d\sigma\right), 
\
\mathcal{H}_0 = L^{2}\left((0,|\partial\Omega|),2ds\right).
\]
Let \(\mathcal{E}_{\varepsilon}\colon \mathcal{H}_{0}\to \mathcal{H}_{\varepsilon}\) be the extension operator defined by
\[
\left(\mathcal{E}_{\varepsilon}u \circ \Phi_{\varepsilon}\right)(s,0)=\left(\mathcal{E}_{\varepsilon}u \circ \Phi_{\varepsilon}\right)(s,1)=\varepsilon u(s)
\
\text{for all\ }s\in\partial\Omega.
\]
Note that
\[
\lim_{\varepsilon\to0}\bigl\|\mathcal{E}_{\varepsilon}u\bigr\|_{\mathcal{H}_{\varepsilon}} = \|u\|_{\mathcal{H}_{0}}.
\]
Hence the family \(\{\mathcal{E}_{\varepsilon}\}_{\varepsilon\in(0,\varepsilon_0)}\) satisfies \eqref{16}, and is an admissible connecting system.

Now we define the operators \(B_\varepsilon\) and \(B_0\). The operator \(B_\varepsilon\) is the resolvent associated with problem \eqref{12}, rescaled by \(\varepsilon\). Namely, define 
\[
B_\varepsilon: \ \mathcal{H}_\varepsilon\rightarrow  \mathcal{H}_\varepsilon,
\]
by
\begin{equation}\label{44}
B_\varepsilon f_{\varepsilon}=\varepsilon u_{\varepsilon}, \ \text{for all} \ f_{\varepsilon}\in \mathcal{H}_\varepsilon,    
\end{equation}
where \(u_{\varepsilon}\) solves \eqref{12} with boundary data \(f_{\varepsilon}\).  A nonzero real \(\lambda(\varepsilon)\) is an eigenvalue of \(B_\varepsilon\) if and only if 
$\lambda_\varepsilon=\frac{\varepsilon(1-b\lambda(\varepsilon))}{\lambda(\varepsilon)}
$ is an eigenvalue of problem \eqref{2} with the same eigenfunction.  Equivalently, \(\lambda(\varepsilon)\neq0\) is an eigenvalue of \(B_\varepsilon\) if and only if $\overline{\lambda}_\varepsilon = \frac{\varepsilon}{\lambda(\varepsilon)}$ is an eigenvalue of problem \eqref{10}.  Since the trace operator \(H^2(\omega_\varepsilon)\to L^2(\partial\omega_\varepsilon)\) is compact, it follows that \(B_\varepsilon\) is compact.

Next, \(B_0\) is the resolvent associated with problem \eqref{14}.  Define
\[
B_0: \ \mathcal{H}_0\rightarrow  \mathcal{H}_0,
\]
by
\begin{equation}\label{45}
B_{0}f=u, \ \ \text{for all} \ f\in \mathcal{H}_0,
\end{equation}
where \(u\) solves \eqref{14} with right‐hand side \(f\).  A nonzero real \(\lambda(0)\) is an eigenvalue of \(B_0\) if and only if 
 $\lambda=\frac{1-b\lambda(0)}{\lambda(0)}
$ is an eigenvalue of problem \eqref{3} with the same eigenfunction.  Moreover,
\[
B_{0}\left(\mathcal{H}_0\right)\subset H^{4}_{p}\bigl((0,|\partial\Omega|)\bigr)
\]
and the embedding
\[
H^{4}_{p}\bigl((0,|\partial\Omega|)\bigr)
\hookrightarrow
\mathcal{H}_0
\]
is compact.  Hence \(B_{0}\) is compact.

\begin{theorem}\label{thm:3.1}
Let \(B_{\varepsilon}\), \(\varepsilon\in[0,\varepsilon_0)\), be defined by \eqref{44}, and let \(B_{0}\) be defined by \eqref{45}.  Then \(B_{\varepsilon}\) converges compactly to \(B_{0}\) as \(\varepsilon\to 0\).        
\end{theorem}
\begin{proof}
Since we only care about the limit \(\varepsilon\to 0\), we may restrict to 
$
\varepsilon \in \bigl[0,\min\{1,\varepsilon_0/2\}\bigr]
$.

Let $f_{\varepsilon}\in \mathcal{H}_\varepsilon$ satisfy
\[
\|f_{\varepsilon}\|_{\mathcal{H}_\varepsilon}=1.    
\]
Then
\[
\begin{aligned}
\frac{1}{\varepsilon^2}&\int_{\partial\omega_{\varepsilon}}f_{\varepsilon}^2 d\sigma
=
\frac{1}{\varepsilon^2}\int_0^{|\partial\Omega|}\tilde f_{\varepsilon}^2(s,0) ds
+\frac{1}{\varepsilon^2}\int_0^{|\partial\Omega|}\tilde f_{\varepsilon}^2(s,1)(1-\varepsilon\kappa(s)) ds=1.
\end{aligned}
\]
Hence, 
$
\varepsilon^{-1}\tilde f_\varepsilon(s,0) \rightharpoonup f_1(s) \ \text{in} \ L^2\left((0,|\partial\Omega|)\right),   
$ and $
\varepsilon^{-1}\tilde f_\varepsilon(s,1) \rightharpoonup f_2(s) \ \text{in} \ L^2\left((0,|\partial\Omega|)\right)  
$ along some subsequence as \(\varepsilon\to 0\).  By Section \ref{ss:3.2}, up to a further subsequence,
$
\tilde u_{\varepsilon}\rightharpoonup u
\ \text{in }H^{2}(\Sigma),
$
where \(u\) solves the limit problem \eqref{43} corresponding to boundary data \(f_{1}+f_{2}\).  

Next, note that
\[
\|B_\varepsilon f_{\varepsilon}\|^2_{\mathcal{H}_\varepsilon}=\|\varepsilon u_{\varepsilon}\|^2_{\mathcal{H}_\varepsilon}
=\int_{\partial\omega_{\varepsilon}}u_{\varepsilon}^2 d\sigma
=\int_0^{|\partial\Omega|}\tilde u_{\varepsilon}^2(s,0) ds
+\int_0^{|\partial\Omega|}\tilde u_{\varepsilon}^2(s,1)(1-\varepsilon\kappa(s)) ds, 
\]
and
\[
\begin{aligned}
<\varepsilon u_{\varepsilon},\mathcal{E}_\varepsilon u>_{\mathcal{H}_\varepsilon}
&=\frac{1}{\varepsilon}\int_{\partial\omega_{\varepsilon}}u_{\varepsilon}\mathcal{E}_\varepsilon u d\sigma
\\&=\int_0^{|\partial\Omega|}\tilde u_{\varepsilon}(s,0)u(s) ds
+\int_0^{|\partial\Omega|}\tilde u_{\varepsilon}(s,1)u(s)(1-\varepsilon\kappa(s)) ds.  
\end{aligned}
\]
Since the trace map is compact and \(\tilde u_{\varepsilon}\rightharpoonup u\) in \(H^2(\Sigma)\), we get
\begin{equation}\label{47}
\|\varepsilon u_{\varepsilon}-\mathcal{E}_\varepsilon u\|^2_{\mathcal{H}_\varepsilon}=
\|\varepsilon u_{\varepsilon}\|^2_{\mathcal{H}_\varepsilon}
-2<\varepsilon u_{\varepsilon},\mathcal{E}_\varepsilon u>_{\mathcal{H}_\varepsilon}
+\|\mathcal{E}_\varepsilon u\|^2_{\mathcal{H}_\varepsilon}\rightarrow 0,
\end{equation}
as $\varepsilon \rightarrow 0$.

Let \(f_{\varepsilon}\in \mathcal{H}_\varepsilon\) and \(f\in \mathcal{H}_0\) satisfy
\begin{equation}\label{48}
\|f_{\varepsilon}-\mathcal{E}_\varepsilon f\|_{\mathcal{H}_\varepsilon}\rightarrow 0 \ \text{as} \ \varepsilon\rightarrow 0.
\end{equation}
Set $f=\mathcal{E}_1 f_1=\mathcal{E}_1 f_2=\varepsilon^{-1}(\mathcal{E}_\varepsilon f_1)(s,0)=\varepsilon^{-1}(\mathcal{E}_\varepsilon f_2)(s,1)$. Then
\[
\begin{aligned}
&\int_0^{|\partial\Omega|}(\varepsilon^{-1}\tilde f_{\varepsilon}(s,0)-f_1(s))^2 ds
+\int_0^{|\partial\Omega|}(\varepsilon^{-1}\tilde f_{\varepsilon}(s,1)-f_2(s))^2(1-\varepsilon \kappa(s)) ds
\\& =\int_0^{|\partial\Omega|}(\varepsilon^{-1}\tilde f_{\varepsilon}-\varepsilon^{-1}(\mathcal{E}_\varepsilon f)(s,0))^2 ds
+\int_0^{|\partial\Omega|}(\varepsilon^{-1}\tilde f_{\varepsilon}(s,1)-\varepsilon^{-1}(\mathcal{E}_\varepsilon f)(s,1))^2(1-\varepsilon \kappa(s)) ds
\\&
=\frac{1}{\varepsilon^2}\int_{\partial\omega_\varepsilon}(f_{\varepsilon}-\mathcal{E}_\varepsilon f)^2 d\sigma\rightarrow 0,
\end{aligned}  
\]
$\text{as} \ \varepsilon\rightarrow 0$.
Hence, $
\varepsilon^{-1}\tilde f_\varepsilon(s,0) \rightharpoonup f_1 \ \text{in} \ L^2\left((0,|\partial\Omega|)\right),   
$ and $
\varepsilon^{-1}\tilde f_\varepsilon(s,1) \rightharpoonup f_2 \ \text{in} \ L^2\left((0,|\partial\Omega|)\right)  
$ as $\varepsilon\to 0$, and by Section \ref{ss:3.2}, up to a subsequence,
$
\tilde u_{\varepsilon}\rightharpoonup u
\ \text{in }H^{2}(\Sigma)
$. The same argument as in \eqref{47} then yields
\[
\|\varepsilon u_{\varepsilon}-\mathcal{E}_\varepsilon u\|_{\mathcal{H}_\varepsilon}\rightarrow 0,
\]
as $\varepsilon \rightarrow 0$. Since \(u = B_{0}f\), this shows
\[
B_{\varepsilon}f_{\varepsilon}
=
\varepsilon u_{\varepsilon}
\rightarrow \mathcal{E}_\varepsilon\bigl(B_{0}f\bigr)
\quad\text{in }\mathcal{H}_\varepsilon,
\]
completing the proof of Theorem~\ref{thm:3.1}.
\end{proof}


\subsection{\large Spectral convergence: asymptotics of eigenvalues}
\noindent
Combining Theorem~\ref{thm:3.1} (compact convergence of \(B_\varepsilon\) to \(B_0\)) with Theorem~\ref{th 2.5} (spectral convergence for compact operators) shows that the spectra of \(B_\varepsilon\) converge to the spectrum of \(B_0\) as \(\varepsilon\to0\).  In particular, if \(\lambda_{k}(\varepsilon)\) denotes the \(k\)-th eigenvalue of \(B_\varepsilon\) (counted in increasing order, counting multiplicities) and \(\lambda_{k}(0)\) denotes the \(k\)-th eigenvalue of \(B_0\), then
\begin{equation}\label{49}
\frac{\lambda_{\varepsilon,k}}{\varepsilon}=\frac{1-b\lambda_k(\varepsilon)}{\lambda_k(\varepsilon)}\sim\frac{1-b\lambda_k(0)}{\lambda_k(0)}=\lambda_k
\end{equation}
as $\varepsilon\rightarrow0$.

Next, let \(u_{\varepsilon,k}\) be an eigenfunction of \eqref{2} corresponding to \(\lambda_{\varepsilon,k}\), normalized so that
$
\|u_{\varepsilon,k}\|_{L^{2}(\partial\omega_{\varepsilon})}
=1.
$
Then \(\overline u_{\varepsilon,k}:=\varepsilon u_{\varepsilon,k}\) is normalized in \(\mathcal{H}_\varepsilon\), i.e.\ 
$
\|\overline u_{\varepsilon,k}\|_{\mathcal{H}_\varepsilon}
=1.
$
By Theorem~\ref{th 2.5} (spectral convergence for \(B_\varepsilon\)), there exists an orthonormal basis \(\{u_k\}_{k\in\mathbb{N}^+}\subset \mathcal{H}_0\) of eigenfunctions of \(B_0\) such that, after passing to a subsequence if necessary,
\begin{equation}\label{vainnikko1}
\|\overline u_{\varepsilon ,k}-\mathcal{E}_\varepsilon u_k\|_{\mathcal{H}_\varepsilon}\rightarrow 0 \ \text{as} \ \varepsilon\rightarrow 0.
\end{equation} 
In order to prove the convergence in $\Sigma$, we argue as before. So, we obtain the problem
\begin{equation}\label{51}
\begin{aligned}
\int_{\Sigma}(D^2u_{\varepsilon,k}&:D^2\varphi) \circ\Phi_\varepsilon(s,t)\,
(1 - \varepsilon t\kappa(s))dtds
\\+&\frac{\mu}{\varepsilon^3}\int^{|\partial\Omega|}_{0}\partial_t(u_{\varepsilon,k}\circ\Phi_\varepsilon)(s,0)\partial_t(\varphi\circ\Phi_\varepsilon)(s,0)ds
\\+&\frac{\mu}{\varepsilon^3}\int^{|\partial\Omega|}_{0}\partial_t(u_{\varepsilon,k}\circ\Phi_\varepsilon)(s,1)\partial_t(\varphi\circ\Phi_\varepsilon)(s,1)(1 - \varepsilon\kappa(s))ds
\\&=\frac{\lambda_{\varepsilon,k}}{\varepsilon}
\int^{|\partial\Omega|}_{0}(u_{\varepsilon,k}\varphi)\circ\Phi_\varepsilon(s,0)ds
+\frac{\lambda_{\varepsilon,k}}{\varepsilon}\int^{|\partial\Omega|}_{0}(u_{\varepsilon,k}\varphi)\circ\Phi_\varepsilon(s,1)(1-\varepsilon\kappa(s))ds.
\end{aligned}
\end{equation} 
If we set in \eqref{51} $\varphi=u_{\varepsilon,k}$, then taking into account \eqref{49} and the normalization $\|u_{\varepsilon,k}\|_{L^2(\partial\omega_{\varepsilon})}=1$, we obtain that
 \(\|\tilde u_{\varepsilon,k}\|_{H^2(\Sigma)}\leq c\). Then, we deduce that there exists a function \(v_k\in H^{2}(\Sigma)\) such that, up to a subsequence,
\[
\tilde u_{\varepsilon,k}\rightharpoonup v_k
\ \text{in }H^{2}(\Sigma),
\ \text{and hence strongly in }H^{1}(\Sigma).
\]

In equation \eqref{51}, by selecting a test function 
\[
\varphi(s,t)=\psi(s),
\ s\in\bigl(0,|\partial\Omega|\bigr),\
\psi\in H^{2}_{p}\bigl(0,|\partial\Omega|\bigr).
\]
and considering \eqref{49} and convergences in Section \ref{ss:3.2}, taking the limit as $\varepsilon \rightarrow 0$, we obtain  
\begin{equation}\label{52}
\int_{\Sigma}\left(
\partial^2_{ss}v_k\partial^2_{ss}\psi
+2\kappa(s)^{2}\partial_{s}v_k\partial_{s}\psi
\right)dsdt
=
2\lambda_k\int^{|\partial\Omega|}_{0}v_k\psi ds.
\end{equation}
Note that all the functions appearing in \eqref{52} are constant in \(t\). Since $\psi\in H^{2}_{p}\bigl(0,|\partial\Omega|\bigr)$ is an arbitrary function, the \eqref{52} is the weak formulation of the following problem: 
\[
\begin{cases}
v^{(4)}_k-2\bigl(\kappa^2v_k'\bigr)'=2\lambda_kv_k,
&\text{in }(0,|\partial\Omega|),\\
v_k^{(l)}(0)=v_k^{(l)}(|\partial\Omega|),
&l=0,1,2,3.
\end{cases}
\]

By re-writing \eqref{vainnikko1} in terms of $u_{\varepsilon , k}$, one can easily deduce that, possibly passing to a subsequence,
$\tilde u_{\varepsilon,k}(s,0)\rightarrow u_k(s)$ in $L^2((0,|\partial\Omega|))$ as $\varepsilon \to 0$, hence $u_k=v_k$ and the proof is complete. 

\begin{remark}
One can also allow the tubular neighborhood \(\omega_{\varepsilon}\) to have variable thickness.  Specifically, fix a smooth function \(g:\partial\Omega\to(0,1)\), and define
\[
\omega_{\varepsilon,g}:=
\bigl\{x\in \omega_{\varepsilon} : 0 < \mathrm{dist}(x,\partial\Omega) < \varepsilon g(s(x))\bigr\},
\]
for all \(\varepsilon\in(0,\varepsilon_0)\), where \(s(x)\) is the nearest point on \(\partial\Omega\) to \(x\). The computations can be carried out exactly as in the previously.  Repeating the same estimates shows that the limit problem becomes
\[
\begin{cases}
(gu'')''-2(\kappa^{2}gu')'
=\lambda u,
& \ \text{in} \ (0,|\partial\Omega|),\\
u^{(k)}(0) = u^{(k)}\bigl(|\partial\Omega|\bigr),
& k = 0,1,2,3,
\end{cases}
\]
in the unknowns \(u\) and \(\lambda\).  For more details on variable‐width limits, see \cite{PDL}, \cite{BD}.    
\end{remark}

\section*{Acknowledgements}
 This research was funded by the Science Committee of the Ministry of Science and Higher Education of the Republic of Kazakhstan (Grant No. BR31714735).
 
\subsection*{Conflict of interest}
This work does not have any conflicts of interest.


\small
\begin{thebibliography}{99}

\bibitem{JMA}
Arrieta, J.M., Carvalho, A.N., Losada-Cruz, G.: Dynamics in dumbbell domains I. Continuity of the set of equilibria. \textit{J. Differ. Equ.} \textbf{231}, 551--597 (2006)

\bibitem{EA}
Arrieta, J.M., López-Fernández, M., Zuazua, E.: Approximating travelling waves by equilibria of non-local equations. \textit{Asymptot. Anal.} \textbf{78}(3), 145--186 (2012)

\bibitem{PDL}
Arrieta, J.M., Ferraresso, F., Lamberti, P.D.: Spectral analysis of the biharmonic operator subject to Neumann boundary conditions on dumbbell domains. \textit{Integr. Equ. Oper. Theory} \textbf{89}, 377--408 (2017)

\bibitem{VP}
Arrieta, J.M., Villanueva-Pesqueira, M.: Elliptic and parabolic problems in thin domains with doubly weak oscillatory boundary. \textit{Commun. Pure Appl. Anal.} \textbf{19}(4), 1891--1914 (2020)

\bibitem{JM}
Arrieta, J.M., Nakasato, J.M., Pereira, M.C.: The $p$-Laplacian equation in thin domains: the unfolding approach. \textit{J. Differ. Equ.} \textbf{274}, 1--34 (2021)

\bibitem{Nakasato}
Arrieta, J.M., Nakasato, J.C., Villanueva-Pesqueira, M., Panasenko, G., Piatnitski, A.: Homogenization in 3D thin domains with oscillating boundaries of different orders. \textit{Nonlinear Anal.} \textbf{251}, 113797 (2025)


\bibitem{bel}
Balinsky, A.A., Evans, W.D., Lewis, R.T.: \textit{The analysis and geometry of Hardy's inequality}. Springer, Cham (2015)

\bibitem{FP}
Borisov, D., Freitas, P.: Asymptotics of Dirichlet eigenvalues and eigenfunctions of the Laplacian on thin domains in $\mathbb{R}^d$. \textit{J. Funct. Anal.} \textbf{258}(3), 893--912 (2010)

\bibitem{BD}
Bucur, D., Henrot, A., Michetti, M.: Asymptotic behaviour of the Steklov spectrum on dumbbell domains. \textit{Commun. Partial Differ. Equ.} \textbf{46}(2), 362--393 (2021)

\bibitem{buopro}
Buoso, D., Provenzano, L.: A few shape optimization results for a biharmonic Steklov problem. \textit{J. Differ. Equ.} \textbf{259}(5), 1778--1818 (2015)

\bibitem{AC}
Carvalho, A.N., Piskarev, S.: A general approximation scheme for attractors of abstract parabolic problems. \textit{Numer. Funct. Anal. Optim.} \textbf{27}(7--8), 785--829 (2006)

\bibitem{CD}
Casado-Diaz, J., Luna-Laynez, M., Suarez-Grau, F.J.: A decomposition result for the pressure of a fluid in a thin domain and extensions to elasticity problems. \textit{SIAM J. Math. Anal.} \textbf{52}(3), 2201--2236 (2020)

\bibitem{Eringen1968}
Eringen, A.C.: Theory of Micropolar Elasticity. In: Liebowitz, H. (ed.) \textit{Fracture}. Academic Press, New York (1968)

\bibitem{Eringen1999}
Eringen, A.C.: \textit{Microcontinuum Field Theories: I. Foundations and Solids}. Springer-Verlag, New York (1999)

\bibitem{fed}
Federer, H.: Curvature measures. \textit{Trans. Amer. Math. Soc.} \textbf{93}, 418--491 (1959)

\bibitem{ferpro}
Ferraresso, F., Provenzano, L.: On the eigenvalues of the biharmonic operator with Neumann boundary conditions on a thin set. \textit{Bull. Lond. Math. Soc.} \textbf{55}(3), 1154--1177 (2023)
\bibitem{SS}
Ferrero, A., Lamberti, P.D.: Spectral stability for a class of fourth order Steklov problems under domain perturbations. \textit{Calc. Var. Partial Differ. Equ.} \textbf{58}, Article 33 (2019)

\bibitem{AFPDL}
Ferrero, A., Lamberti, P.D.: Spectral stability of the Steklov problem. \textit{Nonlinear Analysis} \textbf{222}, 112989 (2022)

\bibitem{G}
Friesecke, G., James, R.D., Müller, S.: A theorem on geometric rigidity and the derivation of nonlinear plate theory from three-dimensional elasticity. \textit{Comm. Pure Appl. Math.} \textbf{55}(11), 1461--1506 (2002)

\bibitem{Antonio}
Gaudiello, A., Panasenko, G., Piatnitski, A.: Asymptotic analysis and domain decomposition for a biharmonic problem in a thin multi-structure. \textit{Commun. Contemp. Math.} \textbf{18}(5), 1550062 (2016)

\bibitem{GA}
Gaudiello, A., Gomez, D., Perez-Martinez, M.-E.: Asymptotic analysis of the high frequencies for the Laplace operator in a thin T-like shaped structure. \textit{J. Math. Pures Appl.} \textbf{134}(9), 299--327 (2020)

\bibitem{Grieser}
Grieser, D.: Thin Tubes in Mathematical Physics, Global Analysis, and Spectral Geometry. In: \textit{Analysis on Graphs and its Applications}, vol. 77, pp. 589--614. American Mathematical Society (2008)

\bibitem{Henry}
Henry, D.: \textit{Perturbation of the boundary in boundary-value problems of partial differential equations}. Cambridge University Press, Cambridge (1981)

\bibitem{PDLLP}
Lamberti, P.D., Provenzano, L.: On the explicit representation of the trace space $H^{3/2}$ and of the solutions to biharmonic Dirichlet problems on Lipschitz domains via multi-parameter Steklov problems. \textit{Rev. Mat. Complut.} \textbf{35}, 53--88 (2022)


\bibitem{lll}
Lewis, R.T., Li, J., Li, Y.: A geometric characterization of a sharp Hardy inequality. \textit{J. Funct. Anal.} \textbf{262}(7), 3159--3185 (2012)

\bibitem{Li}
Li, F., Li, D., Freitas, M.M.: Limiting dynamics for stochastic delay p-Laplacian equation on unbounded thin domains. \textit{Banach J. Math. Anal.} \textbf{18}(2), Article 13 (2024)


\bibitem{Miura}
Miura, T.-H.: Thin-film limit of the Ginzburg-Landau heat flow in a curved thin domain. \textit{J. Differ. Equ.} \textbf{422}, 1--56 (2025)


\bibitem{JC}
Nakasato, J.C., Pazanin, I., Pereira, M.C.: Reaction-diffusion problem in a thin domain with oscillating boundary and varying order of thickness. \textit{Z. Angew. Math. Phys.} \textbf{72}(1), Article 5 (2021)

\bibitem{SA}
Nazarov, S.A., Perez, E., Taskinen, J.: Localization effect for Dirichlet eigenfunctions in thin non-smooth domains. \textit{Trans. Amer. Math. Soc.} \textbf{368}(7), 4787--4829 (2016)


\bibitem{Nogueira}
Nogueira, A., Nakasato, J.C.: The p-Laplacian equation in a rough thin domain with terms concentrating on the boundary. \textit{Ann. Mat. Pura Appl.} \textbf{199}(5), 1789--1813 (2020)


\bibitem{MC}
Pereira, M.C., Rossi, J.D., Saintier, N.: Fractional problems in thin domains. \textit{Nonlinear Anal.} \textbf{193}, 111471 (2020)

\bibitem{FS}
Stummel, F.: Perturbation of domains in elliptic boundary-value problems. In: \textit{Lecture Notes in Mathematics}, vol. 503, pp. 110--136. Springer, Berlin (1976)

\bibitem{GMV}
Vainikko, G.M.: Regular convergence of operators and the approximate solution of equations. \textit{J. Sov. Math.} \textbf{16}, 949--980 (1981)

\end{thebibliography}
\end{document}